\documentclass[a4paper,10pt]{amsart}
\usepackage{amsmath,amsfonts,amssymb,amsthm}
\usepackage{mathrsfs,geometry}
\usepackage[T1]{fontenc}
\usepackage[latin1]{inputenc}
\usepackage{url}
\newcommand{\ninf}{\underset{n\to\infty}{\longrightarrow}}
\newcommand{\mc}{\mathcal}

\newcommand{\f}{\frac}
\newcommand{\ste}{\, ;\,}
\newcommand{\ds}{\displaystyle}
\newcommand{\vfi}{\varphi}
\newcommand{\A}{\mathcal{A}}

\newcommand{\iy}{\ensuremath{\infty}}
\newcommand{\R}{\ensuremath{\mathbf{R}}}
\newcommand{\C}{\ensuremath{\mathbf{C}}}

\newcommand{\N}{\ensuremath{\mathbf{N}}}
\newcommand{\Z}{\ensuremath{\mathbf{Z}}}

\newcommand{\indic}{\ensuremath{\mathbf{1}}}

\renewcommand{\phi}{\varphi}

\renewcommand{\leq}{\leqslant}
\renewcommand{\geq}{\geqslant}

\newcommand{\mcP}{\mathcal{P}}
\newcommand{\mcA}{\mathcal{A}}

\renewcommand{\k}{\kappa} 
\newcommand{\card}{\sharp}
\newcommand{\id}{{id}}
\newcommand{\scalar}[2]{\langle #1 , #2 \rangle}

\newtheorem{theo}{Theorem}
\newtheorem{pr}{Proposition}
\newtheorem{lemma}{Lemma}
\newtheorem{rk}{Remark}
\newtheorem{cor}{Corollary}

\long\def\symbolfootnote[#1]#2{\begingroup
\def\thefootnote{\fnsymbol{footnote}}\footnote[#1]{#2}\endgroup} 

\author[F.~Benaych-Georges]{Florent BENAYCH-GEORGES}
\address{LPMA, UPMC Univ Paris 6, Case courier 188, 4, Place Jussieu, 75252 Paris Cedex 05, France. } \email{florent.benaych@gmail.com}

\author[I.~Nechita]{Ion NECHITA}
\address{Universit\'e de
Lyon, Universit\'e de Lyon 1, Institut Camille Jordan, CNRS UMR
5208, 43, Boulevard du 11 novembre 1918, 69622 Villeurbanne Cedex,
France} 
\email{nechita@math.univ-lyon1.fr}


\begin{document}
\title{A permutation model for free random variables and its classical analogue}

\begin{abstract}
In this paper, we generalize a permutation model for free random variables which was first proposed by Biane in \cite{biane}. We also construct its classical probability analogue, by replacing the group of permutations with the group of subsets of a finite set endowed with the symmetric difference operation. These constructions provide new discrete approximations of the respective free and classical Wiener chaos. As a consequence, we obtain explicit examples of non random matrices which are asymptotically free or independent. The moments and the free (resp. classical) cumulants of the limiting distributions are expressed in terms of a special subset of (noncrossing) pairings. At the end of the paper we present some combinatorial applications of our results.
\end{abstract}

\maketitle

\section*{Introduction}
Free probability is the non-commutative probability theory built upon the notion of independence called freeness.
In classical probability theory, independence characterizes families of random variables whose joint distribution can be deduced from the individual ones by making their tensor product. In the same way, freeness, in free probability theory, characterizes families of random variables whose joint distribution can be deduced from the individual ones by making their free product (with the difference that free random variables belong to non commutative probability spaces, and that their joint distribution is no longer   a probability measure, but  a linear functional on a space of polynomials). Concretely, independent random variables are numbers arising randomly with no influence on each other, whereas free random variables are elements of an operator algebra endowed with a state which do not satisfy any algebraical relation together, as far as what can be observed with the algebra's state is concerned. Free probability theory has been a very active field of mathematics during the last two decades, constructed in a deep analogy with classical probability theory. It follows that there is a kind of dictionary between objects of both theories:  many fundamental notions or results of classical probability theory, like the Law of Large Numbers, Central Limit Theorem, Gaussian distribution, convolution, cumulants, infinite divisibility have a precise analogue in free probability theory. Moreover, several examples of asymptotically free random variables have been found, like random matrices \cite{voic2,voic1,hiai,HT05},  representations of groups \cite{bibi95.01,bibi98.01}, and  a permutation model of P. Biane \cite{biane}. 
In the present paper, we shall firstly generalize this permutation model, and then develop  its analogue from classical probability theory, which will allow us to show that surprisingly,  in the "dictionary" mentioned above between classical and free probability theories, there is  a  correspondence (of minor importance when compared to others, but still interesting) between groups of sets endowed with the symmetric difference operation and groups of permutations,   following from the correspondence between the lattice of partitions and the lattice of non crossing partitions. 

To explain how we construct this model and its analogue from classical probability theory, let us recall a few basic definitions of non commutative probability theory. First of all, let us recall that a {\it non commutative probability space} (as we shall use it) is a complex unital $*$-algebra $\A$ endowed with with   a linear form $\varphi$  such that $\varphi(1)=1$  and for all $x\in \A$, $\vfi(x^*)=\overline{\vfi(x)}$ and $\vfi(xx^*)\geq 0$. The {\it non commutative distribution} of a family $(x_i)_{i\in I}$ of self-adjoint elements of $\A$ is then the application which maps any polynomial $P$ in the non commutative variables $(X_i)_{i\in I}$ to $\vfi(P((x_i)_{ i\in I}))$. This  formalism is the one of free probability theory, but it recovers the one of classical probability theory, because if the algebra $\A$ is commutative, then this distribution is actually the integration with respect to a probability measure on $\R^I$, and $\A$ and $\vfi$ can respectively be identified with a subalgebra of the intersection of the $L^p$ spaces ($p\in [1,+\infty)$) of a certain probability space and with the integration with respect to the probability measure of this probability space. A general example of non commutative probability space of historical importance is, given a countable  group $G$, the $*$-algebra $\C[G]=\C^{(G)}$ 
(the set of finitely supported complex functions on $G$),  endowed with  the notion of adjoint defined by $(\sum_{g\in G}x_g.g)^*=\sum_{g\in G}\overline{x_g}.g^{-1}$ and with the  trace $\vfi(\sum_{g\in G}x_g.g)=x_e$, where $e$ denotes the neutral element of $G$.   
Our asymptotic model for free random variables is constructed in the algebra of the group $\mathcal{S}$ of permutations with finite support of the set of nonnegative integers, whereas its classical probability theory analogue is constructed in the algebra of the group 
$\mathcal{G}$
of finite sets of nonnegative integers endowed with the symmetric difference operation. More 
precisely,
let us define, for all integer $r\geq 1$, and $t\in [0,+\infty)$, the element of $\C[\mathcal{S}]$ 
\[M_r(n,t) = \frac{1}{n^{r/2}}\;\;\; \sum\underbrace{(0 a_1 a_2 \cdots a_r)}_{\substack{\textrm{designs the cycle}\\ 0\to a_1\to a_2\to\cdots \to a_r\to 0}},\]where the sum runs over all $r$-uplets $(a_1,\ldots, a_r)$ of pairwise distinct integers of $[1, nt]$.  In \cite{biane}, it was already proved that the non commutative distribution of the family $(M_1(n,t))_{  t\in [0,+\infty)}$ converges, as $n$ goes to infinity, to the one of a family $(M_1(t))_{  t\in [0,+\infty)}$ which is a free Brownian motion. Here, we shall prove that the non commutative distribution of the family $(M_r(n,t))_{r\geq 1, t\in [0,+\infty)}$ converges, as $n$ goes to infinity, to the one of a family  $(M_r(t))_{ r\geq 1, t\in [0,+\infty)}$ such that    that for all $r,t$, one has $M_r(t)=t^{\f{r}{2}}U_r(t^{-1/2}M_1(t))$, where the $U_r$'s are the Chebyshev polynomials of second kind.

 
The classical probability analogue of this model is constructed 
 replacing the group $\mathcal{S}$ of finitely-supported permutations of the set of nonnegative integers by the group $\mathcal{G}$ of finite sets of nonnegative integers, endowed with the symmetric difference operation (the {\it symmetric difference} $A\Delta B$ of two sets  $A$ and $B$ is  $(A\cup B)\backslash (A\cap B)$), we define, 
 for all integer $r\geq 1$, and $t\in [0,+\infty)$, the element of $\C[\mathcal{G}]$ 
\[L_r(n,t) = \frac{1}{n^{r/2}} \sum\{ a_1, a_2, \ldots ,a_r\},\]where the sum runs over all $r$-uplets $(a_1,\ldots, a_r)$ of pairwise distinct integers of $[1, nt]$.  We shall prove that the non commutative distribution of the family $(L_r(n,t))_{r\geq 1, t\in [0,+\infty)}$ converges, as $n$ goes to infinity, to the one of a family  $(L_r(t))_{ r\geq 1, t\in [0,+\infty)}$ such that  $(L_1(t))_{t\in[0,+\infty)}$ is a classical Brownian motion  and that for all $r,t$, one has $L_r(t)=t^{\f{r}{2}}H_r(t^{-1/2}L_1(t))$, where the $H_r$'s are the Hermite polynomials. 

This model is  constructed on a commutative algebra, hence the joint distribution of the family $(L_r(n,t))_{r\geq 1, t\in [0,+\infty)}$ is the one of a family of classical random variables, and the same holds for the family $(L_r(t))_{r\geq 1, t\in [0,+\infty)}$. This last process is well known in classical probability theory: up to a renormalization, it is the {\it Wiener chaos}  \cite{janson, nualart, nourdin}. Our model provides a new discrete approximation of  the random process $(L_r(t))_{r\geq 1, t\in [0,+\infty)}$. The Wiener chaos is the starting point of a wide construction of stochastic processes, called the {\it stochastic measures} \cite{rota, nualart, peccati}. In  a forthcoming paper, we shall prove that
both our classical and free models can be generalized to this setting.

Let us now go back to the free model and explain how one can obtain non-random asymptotically free matrices. 
From the results stated earlier,
there emerges the general idea that duly renormalized  elements of $\C[\mc{S}]$ of the type $$A(n):=\!\!\!\!\!\!\!\!\!\!\!\!\!\sum_{\substack{a_1,\ldots, a_{r_a}\\ \textrm{in a set of size $n$}}} \!\!\!\!\!\!\!\!\!\!\!\!\!(0a_1\cdots a_{r_a}),\quad B(n):=\!\!\!\!\!\!\!\!\!\!\!\!\!\sum_{\substack{b_1,\ldots, b_{r_b}\\ \textrm{in a set of size $n$}}}\!\!\!\!\!\!\!\!\!\!\!\!\!(0b_1\cdots b_{r_b}),\quad C(n):=\!\!\!\!\!\!\!\!\!\!\!\!\!\sum_{\substack{c_1,\ldots, c_{r_c}\\ \textrm{in a set of size $n$}}} \!\!\!\!\!\!\!\!\!\!\!\!\!(0c_1\cdots c_{r_c}), \text{ etc.} $$ are asymptotically free as $n$ goes to infinity if the respective sets where the $a_i$'s, the $b_i$'s and the $c_i$'s are picked from are pairwise disjoint, and that in this result,  asymptotic freeness is replaced by  asymptotic independence 
if the group $\mc{S}$ of permutations is replaced by the one of finite sets endowed with the symmetric difference operation and cycles $(0x_1\cdots x_r)$ are replaced by sets $\{x_1,\ldots, x_r\}$. 

Let us now comment on Biane's original motivation for this construction. His idea (for $r=1$) easily generalizes for arbitrary $r$. As before, consider a finite set of elements $A(n), B(n), C(n)$, etc. of the group algebra $\C[\mathcal S_N]$, which is possible for $N$ large enough. When viewed as operators on $\mathcal S_N$,  $A(n), B(n), C(n)$, etc. are complex matrices with rows and columns indexed by the elements of $\mathcal S_N$ (these matrices can be seen as   the adjacency matrices of some Cayley graphs). This is the reason why these results provide explicit examples of  asymptotically free families of non random matrices. To our knowledge, there are no other such constructions. The classical probability part of our result also provides an explicit example of commutative family of non random matrices which are asymptotically independent, property   that only random matrices had until now been proved to have. 

In the last part of this paper, 
we shall  explore  connections between several combinatorial structures and the sets of non crossing pairings  which appeared in the formulas of moments and free cumulants in the limit theorems presented above.

\section{The permutation model for free random variables}
\subsection{Computation of the limit distribution}
The non-commutative probability space  we are going to work with is  the group algebra $\C[\mc{S}]$ of the group $\mc{S}$ of finitely-supported permutations of the set of nonnegative integers (i.e. permutations for which all but finitely-many points are fixed points),  
with its canonical trace defined by $\phi(\sum_\sigma x_\sigma \sigma) = x_\id$, where $\id$ is the identity permutation. Let us define, for all integer $r\geq 1$, and $t\in [0,+\infty)$, the element of $\C[\mathcal{S}]$ 
\[M_r(n,t) = \frac{1}{n^{r/2}} \sum(0 a_1 a_2 \cdots a_r),\]where the sum runs over all $r$-uplets $(a_1,\ldots, a_r)$ of pairwise distinct integers of $[1, nt]$.  For $r=0$, we put $M_0(n,t) = \id$.
Our purpose in what follows is to study the asymptotic properties (in the limit $n \to \iy$) of the family $(M_r(n,t))_{r,t}$.

Before stating the main result of this section, let us recall that a free Brownian motion is a process $(S_t)_{t\in [0+\infty)}$ of non commutative random variables with free increments such that for all $t$, $S_t$ is semi-circular with variance $t$. Let us also recall some facts about the Chebyshev polynomials of the second kind, denoted by $(U_n)$.
These are the orthogonal polynomials on $[-2, 2]$ with respect to the semi-circular weight $w(x) = \frac{1}{2\pi}\sqrt{4-x^2}$. They satisfy the property \[U_n(2\cos \theta) = \frac{\sin(n+1)\theta}{\sin \theta}, \quad \forall n \geq 0\]
 and the recurrence relation
\[U_0(x) = 1, \quad
U_1(x) = x, \quad U_1(x)U_n(x) = U_{n-1}(x) + U_{n+1}(x), \forall n\geq 1.\]

\begin{theo}\label{lim.distri.12.01.08}The non commutative distribution of the family $(M_r(n,t))_{r\geq 1, t\in [0,+\infty)}$ converges, as $n$ goes to infinity, to the one of a family  $(M_r(t))_{ r\geq 1, t\in [0,+\infty)}$ such that   $(M_1(t))_{  t\in [0,+\infty)}$  is a free Brownian motion and for all $r,t$, one has $M_r(t)=t^{\f{r}{2}}U_r(t^{-1/2}M_1(t))$, where the $U_r$'s are the Chebyshev polynomials of second kind.
\end{theo}

\begin{proof}{\it Step I. }  It follows from a direct application of Theorem 1 of \cite{biane} that the non commutative distribution of the family $(M_1(n,t))_{  t\in [0,+\infty)}$ converges, as $n$ goes to infinity, to the one of a family $(M_1(t))_{  t\in [0,+\infty)}$ which is a free Brownian motion. 

{\it Step II. } Let us prove that for all integer $r \geq 1$, and $t\in (0,+\infty)$, 
\begin{equation}\label{07.01.08.1}\lim_{n \to \iy} \phi[(M_1(n,t)M_r(n,t) - tM_{r-1}(n,t) - M_{r+1}(n,t))^2] = 0.\end{equation}
We first compute $M_1(n,t)M_r(n,t)$:

\begin{eqnarray*}M_1(n,t)M_r(n,t) &=& n^{- \frac{r+1}{2}}\!\!\!\!\!\!\sum_{\substack{(a_1, \ldots, a_r)\\ (a_{r+1})}}\!\!\!\!\!\! (0a_{r+1})(0a_1a_2 \cdots a_r)\\
&=& n^{- \frac{r+1}{2}} \!\!\!\!\!\!\!\!\sum_{(a_1, \ldots, a_{r+1})}\!\!\!\!\!\!\!\! (0a_1a_2 \cdots a_ra_{r+1}) + n^{- \frac{r+1}{2}} \sum_{k=1}^{r}\sum_{(a_1, \ldots, a_r)} (0a_1a_2 \cdots a_{k-1})(a_k \cdots a_r)\\ &=&
M_{r+1}(n,t) + \f{\lfloor nt\rfloor}{n}M_{r-1}(n,t) + n^{- \frac{r+1}{2}} \sum_{k=1}^{r-1}\sum_{(a_1, \ldots, a_r)} (0a_1a_2 \cdots a_{k-1})(a_k \cdots a_r).
\end{eqnarray*}
Thus, it suffices to show that ($a=(a_1, \ldots, a_r), b=(b_1, \ldots, b_r)$)
\[\lim_{n \to \iy} \phi[(n^{- \frac{r+1}{2}} \sum_{k=1}^{r-1}\sum_a (0a_1a_2 \cdots a_{k-1})(a_k \cdots a_r))^2] = 0.\]
But
\[\left(\sum_{k=1}^{r-1}\sum_a (0a_1a_2 \cdots a_{k-1})(a_k \cdots a_r)\right)^2 = \sum_{k,l=1}^{r-1}\sum_{a, b} (0a_1a_2 \cdots a_{k-1})(a_k \cdots a_r)(0b_1b_2 \cdots b_{l-1})(b_l \cdots b_r)\]
In order for the permutation on the right-hand side to be the identity, it has to be that
\[(0b_1b_2 \cdots b_{l-1})(b_l \cdots b_r) = [(0a_1a_2 \cdots a_{k-1})(a_k \cdots a_r)]^{-1} = (a_ka_ra_{r-1} \cdots a_{k+1})(0a_{k-1} \cdots a_1)\]
and thus $k=l$ and the $b$'s are determined (modulo some circular permutation of size at most $r$) by the $a$'s. We find that there are at most $(r-1)r!(nt)^r$ terms in the sum which are equal to the identity and \eqref{07.01.08.1} follows.

{\it Step III. } To prove the existence of a limit to the non commutative distribution of  the family $(M_r(n,t))_{r\geq 1, t\in [0,+\infty)}$, we have to prove that for all polynomial $P$ in the non commutative variables $(X_r(t))_{r\geq 0,t\in [0,+\infty)}$,$$\vfi(P((M_r(n,t))_{r\geq 0,t\in [0,+\infty)}))$$   has a finite limit as $n$ goes to infinity.
First of all, by linearity,  can suppose that $P$ is a monomial $X_{r_1}(t_1)\cdots X_{r_k}(i_k)$ with $r_1,\ldots, r_k\geq 0, t_1,\ldots, t_k\in [0,+\infty) $. Then let us prove it by induction on $R:=\max\{r_1,\ldots, r_k\}$. If $R=0$ or $1$, it follows from the first step of the proof and the convention  $M_0(n,t)=1$. Now, let us suppose the result to be proved until rank $R-1$. Replacing, for all $t\in [0,+\infty)$, each $X_R(t)$ in $P$ by $$(X_1(t)X_{R-1}(t)-tX_{R-2}(t))-(X_1(t)X_{R-1}(t)-tX_{R-2}(t)-X_R(t))$$ and using the second  step of the proof with the Cauchy-Schwarz inequality, one gets the convergence. Let us denote the limit distribution by $\Psi : \C\langle X_r(t)\ste  r\geq 0,t\in [0,+\infty)\rangle\to \C$. 

 {\it Step IV. } Now, it remains only to identify the limit distribution. Note first that by the first step and the convention  $M_0(n,t)=1$, the Cauchy-Schwarz inequality allows us to claim that the bilateral ideal generated by $$\{X_0(t)-1\ste t\in [0,+\infty) \}\cup \{X_1(t)X_{m-1}(t)-tX_{m-2}(t)-X_m(t)\ste m\geq 2, t\in [0,+\infty)\}$$ is contained in the kernel of $\Psi$. Hence up to a quotient of the algebra $\C\langle X_r(t)\ste r\geq 0, t\in [0,+\infty)\rangle$, one can suppose that for all $m\geq 2$, $t\in [0,+\infty)$, $X_0(t)=1$ and $ X_1(t)X_{m-1}(t)=tX_{m-2}(t)+X_m(t)$. It allows us to claim that for all  $m\geq 0$,  $t\in [0,+\infty)$, $X_m(t)$ is a polynomial in $  X_1(t)$, namely that $X_m(t)=t^{\f{m}{2}}U_m(t^{-1/2}X_1(t))$, where  the $U_m$'s are the Chebyshev polynomials of second kind (indeed, this family is completely determined by the fact that $U_0=1, U_1=X$ and for all $m\geq 2$, $U_1U_{m-1}=U_{m-2}+U_m$). Since by the first step, $(M_1(t))_{  t\in [0,+\infty)}$  is a free Brownian motion,
 the proof is complete.
\end{proof}

The following corollary generalizes  Theorem 1 of \cite{biane}. Roughly speaking, it states that duly renormalized  elements of $\C[\mc{S}]$ of the type $$A(n):=\!\!\!\!\!\!\!\!\!\!\!\!\!\sum_{\substack{a_1,\ldots, a_{r_a}\\ \textrm{in a set of size $n$}}} \!\!\!\!\!\!\!\!\!\!\!\!\!(0a_1\cdots a_{r_a}),\quad B(n):=\!\!\!\!\!\!\!\!\!\!\!\!\!\sum_{\substack{b_1,\ldots, b_{r_b}\\ \textrm{in a set of size $n$}}}\!\!\!\!\!\!\!\!\!\!\!\!\!(0b_1\cdots b_{r_b}),\quad C(n):=\!\!\!\!\!\!\!\!\!\!\!\!\!\sum_{\substack{c_1,\ldots, c_{r_c}\\ \textrm{in a set of size $n$}}} \!\!\!\!\!\!\!\!\!\!\!\!\!(0c_1\cdots c_{r_c}), \text{ etc.} $$ are asymptotically free as $n$ goes to infinity if the respective sets where the $a_i$'s, the $b_i$'s and the $c_i$'s are picked are pairwise disjoint. Biane had proved it in the case where $r_a=r_b=r_c=\cdots=1$.

 \begin{cor}\label{ok-decay.13.01.08}Fix $p\geq 1$, $r_1,\ldots, r_p\geq 0$, $t_0<t_1<\cdots <t_p$, and define, for all $i=1,\ldots, p$, for each $n\geq 1$, $M_i(n)=\ds n^{-\f{r_i}{2}}\sum (0a_1\cdots a_{r_i})$, where the sum runs over all $r_i$-uplets $(a_1,\ldots, a_{r_i})$ of pairwise distinct integers of $(nt_{i-1},nt_i]$.  Then $M_1(n),\ldots, M_p(n)$ are asymptotically free as $n$ goes to infinity.
\end{cor}

\begin{proof}Let us define, for all $i=1,\ldots, p$ and $n\geq 1$, $S_i(n):=\ds n^{-\f{1}{2}}\sum_{\substack{a\in (nt_{i-1},nt_i]\\ \textrm{$a$ integer}}} (0a)$. By the previous theorem,  as $n$ goes to infinity, the non commutative distribution of $(S_1(n),\ldots, S_p(n))$ tends to the one of a free family $(s_1,\ldots, s_p)$ of semi-circular elements (with various variances). Moreover, the same theorem says that for all $i$, as $n$ goes to infinity, $$\ds \lim_{n\to \infty}\vfi((M_i(n)-(t_i-t_{i-1})^{\f{r_i}{2}}U_{r_i}(\sqrt{t_i-t_{i-1}}S_i(n))^2)= 0.$$It follows that the non commutative distribution of the family $$(S_1(n), M_1(n),\ldots, S_p(n), M_p(n))$$ converges to the one of $$(s_1, (t_1-t_{0})^{\f{r_1}{2}}U_{r_1}(\sqrt{t_1-t_{0}}s_1),\ldots, s_p, (t_p-t_{p-1})^{\f{r_p}{2}}U_{r_1}(\sqrt{t_p-t_{p-1}}s_p),$$ which finishes the proof.
\end{proof}

\subsection{Moments and cumulants of the limit distribution}
We now turn to the moments and the free cumulants of the family $(M_r(t))_{ r\geq 1, t\in [0,+\infty)}$. As we shall see, these quantities have elegant closed expressions in terms of non-crossing pairings of a special kind. Let us now introduce the combinatorial objects of interest. For $f$ function defined on a finite set $X$, $\ker f$ designates the partition of $X$ by the level sets of $f$.  For every $p \geq 1$ and for every vector $r = (r_1, \ldots, r_p)$ of positive integers, consider the function $f_r:\{1, \ldots |r|\} \to \{1, \ldots p\}$ defined by $f_r(x) = k$ if and only if  $r_1+\cdots + r_{k-1} < x \leq r_1+\cdots +r_{k}$ (we put
$|r| = r_1 + \cdots + r_p$). We introduce the set $NC_2(r)$ of non-crossing pairings $\pi$ of the set $\{1, \ldots, |r|\}$ which do not link two elements who have the same image by $f_r$, i.e. such that $\pi \wedge \hat 1_r = \hat 0_{|r|}$, where $\hat 1_r=\ker f_r$ and $\hat 0_{|r|}$ is the singletons partition of $\{1, \ldots, |r|\}$. 
 We also introduce $NC_2^*(r) = \{\pi \in NC_2(r) | \pi \vee \hat 1_r = \hat 1_{|r|}\}$, where $\hat 1_{|r|}$ is the one-block-partition of $\{1, \ldots, |r|\}$. For $s$ positive integer, we denote with  $\langle s\rangle_p = (s, s, \ldots, s)$ the constant vector where $s$ appears $p$ times. 

In the following theorem,  we compute the  mixed moments and free cumulants of the family $(M_r)_{r\geq 1}=(M_r(1))_{r\geq 1}$ (the mixed moments and cumulants  of the family $(M_r(t))_{r\geq 1, t\in [0,+\infty)}$ can easily be computed in the same way).
 \begin{theo}\label{th:free}
The distribution of the family  $(M_r)_{r \geq 1}$ is characterized by the fact that its mixed moments are given by
	\[\phi(M_{r_1}M_{r_2} \cdots M_{r_p}) = \card NC_2 (r)\]
and its free cumulants are given by
	\[\k_p(M_{r_1}, M_{r_2}, \ldots ,M_{r_p}) = \card NC_2^* (r).\]	
\end{theo}

\begin{rk}
Although they are clearly dependent, the elements $M_r$ are not correlated: 
$\phi(M_qM_r) = 0$
if $q \neq r$ (this follows from the orthogonality of the Chebyshev polynomials).
\end{rk}

\begin{rk}This theorem provides a new proof (even though there are already many !) of the formula of the free cumulants of the free Poisson distribution (also called Marchenko-Pastur distribution, see \cite{hiai}). Indeed, $M_2+1=M_1^2$ is well known to have a free Poisson distribution with mean $1$, all of whose cumulants except the first one the same as the free cumulants of $M_2$. By the theorem, for all $p\geq 2$, $\kappa_p(M_2)$ is the cardinality of $\{\pi \in NC_2(2p)| \pi \vee \hat 1_{\langle 2 \rangle_p} = \hat 1_{2p}\} $. In \cite{nica-speicher}, it is shown that
\[\{\pi \in NC(2p)| \pi \vee \hat 1_{\langle 2 \rangle_p} = \hat 1_{2p}\} = \{\pi \in NC(2p)| 1 \stackrel{\pi}{\sim}2p, 2i \stackrel{\pi}{\sim}2i +1, \forall i \in \{1, \ldots, p-1\}\}.\]
Thus, 
\[\{\pi \in NC_2(2p)| \pi \vee \hat 1_{\langle 2 \rangle_p} = \hat 1_{2p}\} = \{\,\{\,\{2p,1\},\,\{2,3\},\, \ldots, \,\{2p-2, 2p-1\}\,\}\,\},\]
which is a partition of $NC_2(\langle 2 \rangle_p)$,   hence $\k_p(M_2) = 1$.
\end{rk}
 
\begin{proof} Let us first prove that the mixed moments are given by the formula of the theorem. 
Using the identity $(0b_1b_2 \cdots b_s) = (0b_s)(0b_{s-1}) \cdots (0b_1)$, we have
\[\prod_{j=1}^{p}{M_{r_j}(n,1)} = n^{-\frac{|r|}{2}}\sum_{a}(0a_1)(0a_2) \cdots (0a_{|r|}),\]
where the sum is taken over all families $a = (a_1, \ldots a_{|r|}) \in \{1,\ldots, n\}^{|r|}$ such that for all $k,l \in \{1, \ldots |r|\}$,  $a_k \neq a_l$ whenever $f_r(k) = f_r(l)$. To such a family $a$ we associate the partition $\mcP(a)$ of the set $\{1, \ldots |r|\}$ defined by $k \sim l $ if and only if  $a_k = a_l$. Thus, for all $a$, $\mcP(a)$ does not link two elements that have the same image by $f_r$, i.e. satisfies  $\mcP(a) \wedge \hat 1_r = \hat 0_{|r|}$. We regroup the terms of the preceding sum according to the partitions $\mcP$:
\[\sum_\pi n^{-\frac{|r|}{2}}\sum_{a : \mcP(a) = \pi}(0a_1)(0a_2) \cdots (0a_{|r|}).\]
Let us show that among the partitions $\pi$ such that $\pi \wedge \hat 1_r = \hat 0_{|r|}$, the only partitions  that contribute to the limit,  as $n$ goes to infinity, 
are non-crossing pairings, i.e. elements of $NC_2(r)$. If $\pi = \mathcal P(a)$ contains a singleton $\{k\}$, then the permutation  $(0a_1)(0a_2) \cdots (0a_{|r|})$ cannot be the identity, because the element $a_k$ appears only once and thus its image cannot be itself. Consider now a partition $\pi$ with no singleton but with a class with at least three elements. It is easy to show that there are no more than $n^{\frac{|r|-1}{2}}$ families $a$ such that $\mcP(a) = \pi$ and thus they have no contribution asymptotically. We have shown that only pairings contribute to the trace. The argument in \cite{biane}, Lemma 2 (which adapts \emph{mutatis mutandis} to our case) shows that only the non-crossing pairings contribute, completing the proof.
 
Let us now compute the free cumulants. 
To a pairing $\mcP \in NC_2(r)$ we associate the non-crossing partition $\bar \mcP \in NC(p)$ which encodes the way $\mcP$ links the blocks of $\hat 1_r$ : $k \stackrel{\bar \mcP}{\sim} l$ if and only if  $r_1+ \cdots + r_k \stackrel{\mcP \vee \hat 1_r}{\sim} r_1+ \cdots + r_l$, for all $k,l \in \{1, \ldots, p\}$.We have$$\phi((M_{r_1}M_{r_2} \cdots M_{r_p}) =\card  NC_2(r)=\sum_{\pi \in NC(p)} \card \{\mcP \in NC_2(r) | \bar \mcP = \pi\}.$$Since the functionals $NC(p) \ni \pi \mapsto  \card \{\mcP \in NC_2(r) | \bar \mcP = \pi\}$ are multiplicative, we have identified the free cumulants of the family $(M_r)_{r \geq 1}$:
\[\forall p\geq 1, r_1,\ldots, r_p\geq 1, \; \k_\pi(M_{r_1},M_{r_2}, \ldots, M_{r_p}) = \card \{\mcP \in NC_2(r) | \bar \mcP = \pi\}.\]
Considering the case $\pi = \hat 1_p$, we obtain the announced formula for the free cumulants.
\end{proof}

\subsection{An application: linearization coefficients for orthogonal polynomials}
As a corollary of Theorems \ref{lim.distri.12.01.08} and  \ref{th:free}, we recover some formulas already obtained in \cite{anshelevich} using different techniques. Consider a family $(P_n)$ of orthonormal polynomials with respect to some weight $w$. For an integer vector $r = (r_1, \ldots, r_p)$ there is a decomposition
\[P_{r_1}(x)P_{r_2}(x) \cdots P_{r_p}(x) = \sum_{k=0}^{|r|} c_k^{(r)} P_k(x),\]
where the scalars $c_k^{(r)} \in \R$ are called \emph{linearization coefficients} of the family $(P_n)$. They can easily be recovered by integration:
\[c_k^{(r)} = \int P_{r_1}(x)P_{r_2}(x) \cdots P_{r_p}(x) \cdot P_k(x) dw(x).\]
For the Chebyshev polynomials, these integrals are the expectation (the trace) of the corresponding products of the random variables $M_r$:
\begin{cor}\label{linearisation.Cheb}
The linearization coefficients for the Chebyshev polynomials of the second kind $U_n$are given by
\begin{equation}\label{cheb.02.02.09} c_k^{(r)} = \card NC_2(r \cup k),\end{equation}
where $r \cup k$ is the vector $(r_1, \ldots, r_p, k)$.
\end{cor}

\begin{rk}It was mentioned to us by a referee that formula \eqref{cheb.02.02.09} had already been proved, with another method, in \cite[Th. 7]{viennot}.\end{rk}

In \cite{anshelevich}, a similar formula is deduced for the centered free Charlier polynomials $V_n$. These polynomials are orthogonal with respect to the centered Marchenko-Pastur density \[w_2(t) =\indic_{]-1,3]}(t) \frac{1}{2\pi} \sqrt{\frac{4}{1+t}-1} .\]
Note that $M_2=M_1^2-1$ has the distribution $d\mu_2 = w_2(t) dt$. Moreover, one can easily see that $V_n \circ U_2 = U_{2n}$ and thus
\[\int V_{r_1}(x)V_{r_2}(x) \cdots V_{r_p}(x) \cdot V_k(x) dw_2(x) = \int U_{2r_1}(x)U_{2r_2}(x) \cdots U_{2r_p}(x) \cdot U_{2k}(x) dw(x).\]
We obtain
\begin{cor}\label{linearisation.Char} The linearization coefficients for the centered free Charlier polynomials $V_n$are given by
\[d_k^{(r)} = \card NC_2(2r \cup 2k),\]
where $2r \cup 2k$ is the vector $(2r_1, \ldots, 2r_p, 2k)$.
\end{cor}
Using the bijection between non-crossing pairings of size $2n$ and non-crossing partitions of size $n$ (see \cite{nica-speicher}, pp. 153--154), one can easily see that the sets $NC_2(2r \cup 2k)$ and $\{\pi \in NC(r \cup k | \pi \text{ has no singleton}\}$ have the same cardinality, hence our formula is equivalent to the one in \cite{anshelevich}.

\section{A classical probability analogue}
The model we have studied in the first part
involves permutations, asymptotical freeness, non-crossing pairings, the semi-circular distribution and its orthogonal polynomials, the second kind Chebyshev polynomials. By replacing permutations with sets, we construct in this section an analogue model, where the objects from free probability are replaced by their classical counterparts, respectively independence, (possibly crossing) pairings, and the Gaussian distribution with the orthogonal Hermite polynomials.

\subsection{Computation of the limit distribution}
The non-commutative probability space  we are going to work with here is  the group algebra $\C[\mc{G}]$ of the group $\mc{G}$ of finite sets of nonnegative integers endowed with he symmetric difference operation,  
with its canonical trace defined by $\psi(\sum_A x_ A A) = x_\emptyset$. Let us define, for all integer $r\geq 1$, and $t\in [0,+\infty)$, the element of $\C[\mathcal{G}]$ 
\[L_r(n,t) = \frac{1}{n^{r/2}} \sum\{a_1, a_2, \ldots, a_r\},\]where the sum runs over all $r$-uplets $(a_1,\ldots, a_r)$ of pairwise distinct integers of $[1, nt]$.  For $r=0$, we put $L_0(n,t) = \emptyset$ (which is the unity of this algebra).
Our purpose in what follows is to study the asymptotic properties (in the limit $n \to \iy$) of the family $(L_r(n,t))_{r,t}$.

Recall that for every $p \geq 1$ and for every vector $r = (r_1, \ldots, r_p)$ of positive integers, the function $f_r:\{1, \ldots |r|\} \to \{1, \ldots p\}$ is the projection defined by $f_r(x) = k$ iff $r_1+\cdots + r_{k-1} < x \leq r_1+\cdots +r_{k}$ ($|r| = r_1 + \cdots + r_p$). We replace the non-crossing partitions from the free case with general partitions: $\Pi_2(r)$ is the set of pairings $\pi$ of $\{1, \ldots, |r|\}$ which do not link two elements who have the same image by $f_r$, i.e. such that $\pi\wedge \hat 1_r = \hat 0_{|r|}$, where $\hat 1_r$ is still the partition of $\{1, \ldots, |r|\}$ with blocks $f_r^{-1}(1), f_r^{-1}(2), \ldots, f_r^{-1}(p)$.  We also introduce $\Pi_2^*(r) = \{\pi \in \Pi_2(r) | \pi \vee \hat 1_r = \hat 1_{|r|}\}$. 

In the following lemma we compute the asymptotic joint moments of the random variables $L_r(n,t)$.
\begin{lemma}\label{12.01.08.20h47}
Let $p \geq 1$ and consider $t_1,\ldots, t_p>0$ and a family of positive integers $r = (r_1, \ldots, r_p)$. Then, in the limit $n \to \iy$, the trace $\psi\left[\prod_{j=1}^{p}{L_{r_j}(n,t_j)}\right]$ converges to $$\ds\sum_{\pi\in \Pi_2(r)}\prod_{\{i,j\}\in \pi}\min(t_{f_r(i)},t_{f_r(j)}).$$
\end{lemma}
\begin{proof}
Using the properties of the symmetric difference $\Delta$, we get
\[\prod_{j=1}^{p}{L_{r_j}(n,t_j)} = n^{-\frac{|r|}{2}}\sum_{a}\{a_1\}\Delta \{a_2\}\Delta \cdots \Delta \{a_{|r|}\},\]
where the sum is taken over all families $a = (a_1, \ldots a_{|r|})$ of positive integers  such that for all $k,l \in \{1, \ldots |r|\}$, $a_k \in [1,nt_{f_r(k)}]$ and $a_k \neq a_l$ whenever $f_r(k) = f_r(l)$. To such a family $a$ we associate the partition $\mcP(a)$ of the set $\{1, \ldots |r|\}$ defined by $k \sim l $ if and only if $a_k = a_l$. Thus, for all $a$, $\mcP(a)$ does not link two elements that have the same image by $f_r$. We regroup the terms of the preceding sum according to the partitions $\mcP$:
\[\sum_\pi n^{-\frac{|r|}{2}}\sum_{a : \mcP(a) = \pi}\{a_1\}\Delta \cdots \Delta \{a_{|r|}\}.\]
Let us show that only pairings can contribute to the asymptotic trace of the sum. It is obvious that $\{a_1\}\Delta \cdots \Delta \{a_{|r|}\}$ is the empty set if and only if each $a_i$ appears an even number of times. Thus, $\pi = \mathcal P(a)$ cannot contain singletons. On the other hand, if $\pi$ contains no singleton but has a class with at least three elements, it is easy to show that there are no more than $(n\max\{t_1,\ldots, t_p\})^{\frac{|r|-1}{2}}$ families $a$ such that $\mcP(a) = \pi$ and thus such partitions $\pi$ do not contribute asymptotically. 

For $\pi$ pairing of $\Pi_2(r)$, the number of families $a$ such that $\mathcal P(a) = \pi$, is 
asymptotic, 
 as $n$ goes to infinity, to $n^{\frac{|r|}{2}}\prod_{\{i,j\}\in \pi} \min(t_{f_r(i)},t_{f_r(j)}) ,$
which concludes the proof.
\end{proof}

Before stating the main result of this section, let us recall  some facts about the Hermite polynomials, denoted by $(H_n)$.
These are the orthogonal polynomials on the real line with respect to the standard Gaussian measure. They satisfy  the recurrence relation
\[H_0(x) = 1, \quad
H_1(x) = x, \quad H_1(x)H_r(x) = H_{r+1}(x) + rH_{r-1}(x), \forall r\geq 1.\]

\begin{theo}\label{lim.distri.12.01.08.classique}The distribution of the family $(L_r(n,t))_{r\geq 1, t\in [0,+\infty)}$ converges, as $n$ goes to infinity, to the one of a commutative family  $(L_r(t))_{ r\geq 1, t\in [0,+\infty)}$ such that   $(L_1(t))_{  t\in [0,+\infty)}$  is a classical Brownian motion and for all $r,t$, one has $L_r(t)=t^{\f{r}{2}}H_r(t^{-1/2}L_1(t))$, where the $H_r$'s are the Hermite polynomials.
\end{theo}

\begin{proof}{\it Step 0. } Note first that the symmetric difference is a commutative operation on sets. Hence the algebra $\C[\mc{G}]$ is commutative.

{\it Step I. }  It follows from a direct application of the previous lemma that the  distribution of the family $(L_1(n,t))_{  t\in [0,+\infty)}$ converges, as $n$ goes to infinity, to the one of a classical Brownian motion $(L_1(t))_{  t\in [0,+\infty)}$. 

{\it Step II. } Let us prove that for all integer $r \geq 1$, and $t\in (0,+\infty)$, 
\begin{equation}\label{12.01.08.1}\lim_{n \to \iy} \psi[(L_r(n,t)L_1(n,t) - rtL_{r-1}(n,t) - L_{r+1}(n,t))^2] = 0.\end{equation}
This is a consequence of
the following computation of  $L_r(n,t)L_1(n,t)$. The sums run over integers of  $[1,nt]$.
\begin{eqnarray*}
L_r(n,t)L_1(n,t)&=& n^{- \frac{r+1}{2}} \sum_{\substack{(a_1, \ldots, a_r)\\ (a_{r+1})}} \{a_1\} \Delta \cdots \Delta \{a_{r+1}\} \\
&=& n^{- \frac{r+1}{2}} \sum_{(a_1, \ldots, a_{r+1})}  \{a_1,a_2, \ldots a_r,a_{r+1}\} \\ &&+ n^{- \frac{r+1}{2}} \sum_{k=1}^{r}\sum_{(a_1, \ldots, a_r)} \{a_1, a_2, \ldots, \check{a}_k, \ldots, a_r\}\\
&=& L_{r+1}(n,t) + n^{- \frac{r+1}{2}} \sum_{k=1}^{r}(\lfloor nt\rfloor -r+1)\sum_{(b_1, \ldots, b_{r-1})} \{b_1, b_2, \ldots, b_{r-1}\}\\
&=& L_{r+1}(n,t) + \frac{\lfloor nt\rfloor-r+1}{n} rL_{r-1}(n,t)\\ & =& L_{r+1}(n,t) + rtL_{r-1}(n,t) +\epsilon_n L_{r-1}(n,t) \textrm{, with $\epsilon_n\ninf 0$}.
\end{eqnarray*}

{\it Step III } and {\it Step IV }are as in the proof of Theorem \ref{lim.distri.12.01.08}, with the difference that here, the algebra is commutative, hence one-dimensional non-commutative distributions are integrations with respect to a probability measure, which is unique in this case.    \end{proof}

The following corollary is the classical probability theory counterpart of corollary   \ref{ok-decay.13.01.08}.  Roughly speaking, it states that duly renormalized  elements of $\C[\mc{G}]$ of the type $$A(n):=\!\!\!\!\!\!\!\!\!\!\!\!\!\sum_{\substack{a_1,\ldots, a_{r_a}\\ \textrm{in a set of size $n$}}} \!\!\!\!\!\!\!\!\!\!\!\!\!\{a_1,\ldots, a_{r_a}\},\quad B(n):=\!\!\!\!\!\!\!\!\!\!\!\!\!\sum_{\substack{b_1,\ldots, b_{r_b}\\ \textrm{in a set of size $n$}}}\!\!\!\!\!\!\!\!\!\!\!\!\!\{b_1,\ldots, b_{r_b}\},\quad C(n):=\!\!\!\!\!\!\!\!\!\!\!\!\!\sum_{\substack{c_1,\ldots, c_{r_c}\\ \textrm{in a set of size $n$}}} \!\!\!\!\!\!\!\!\!\!\!\!\!\{c_1,\ldots c_{r_c}\},\ldots $$ are asymptotically independent as $n$ goes to infinity if the respective sets where the $a_i$'s, the $b_i$'s and the $c_i$'s are picked are pairwise disjoint. 

 \begin{cor}Fix $p\geq 1$, $r_1,\ldots, r_p\geq 0$, $t_0<t_1<\cdots <t_p$, and defines, for all $i=1,\ldots, p$, for each $n\geq 1$, $L_i(n)=\ds n^{-\f{r_i}{2}}\sum \{a_1,\ldots, a_{r_i}\}$, where the sum runs over all $r_i$-uplets $(a_1,\ldots, a_{r_i})$ of pairwise distinct integers of $(nt_{i-1},nt_i]$.  Then $L_1(n),\ldots, L_p(n)$ are asymptotically independent as $n$ goes to infinity.
\end{cor}

\begin{proof}{\it Mutatis mutandis,} the proof goes along the same lines as the one of corollary \ref{ok-decay.13.01.08}.
\end{proof}

\subsection{Moments and cumulants of the limit distribution}

In the following theorem,  we compute the  mixed moments and cumulants of the family $(L_r)_{r\geq 1} = (L_r(1))_{r\geq 1}$ (the mixed moments and cumulants  of the family $(L_r(t))_{r\geq 1, t\in [0,+\infty)}$ can easily be computed in the same way). Here, the analogy with the free probability model is obvious, since the formulas are the same ones as in Theorem \ref{th:free}, with the difference that the pairings are now allowed to have crossings.
 \begin{theo}\label{th:classique.12.01.08}
The distribution of the family  $(L_r)_{r \geq 1}$ is characterized by the fact that its mixed moments are given by
	\[\psi(L_{r_1}L_{r_2} \cdots L_{r_p}) = \card \Pi_2 (r)\]
and its   classical cumulants are given by
	\[k_p(L_{r_1}, L_{r_2}, \ldots ,L_{r_p}) = \card \Pi_2^* (r).\]	
\end{theo}

\begin{proof} The moments have been computed in Lemma \ref{12.01.08.20h47} and the cumulants can be computed in the same way as  in the proof of Theorem \ref{th:free}.
\end{proof}

\begin{rk}
The correspondence between the limit distributions of the classical and the free case is not the Bercovici-Pata bijection, since the distribution of $L_2$ is not a classical Poisson distribution.
\end{rk}

\subsection{An application: linearization coefficients for orthogonal polynomials} As in Corollaries \ref{linearisation.Cheb} and \ref{linearisation.Char}, one  deduce from this work  combinatorial formulas for  the  linearization coefficients for Hermite and centered Charlier polynomials. Up to normalization, the formulas are the same ones, with the difference that non crossing parings are replaced by pairings.

\section{Further combinatorics}In this section, we explore  connections between several combinatorial structures and the sets $NC_2(r)$ and $NC_2^*(r)$, which appeared in the formulas of moments and free cumulants of the family $M_r(1)$. 

\subsection{A bijection with a class of paths} Here, we shall denote the set of nonnegative integers by $\N$ and the set of integers by $\Z$. 

It is well known that for all $n\geq 1$, the  $n$-th moment of a semi-circular element is the number of Dyck paths with length $n$, i.e. of functions $\gamma : \{0,\ldots, n\}\to \N$  such that  $\gamma(0)=\gamma(n)=0$ and for all $i$, $|\gamma(i)-\gamma(i-1)|=1$. 
Since for $n,t$ fixed, the  $M_r(n,t)$'s ($r\geq 1$) are a generalizations of the Jucys-Murphy element $M_1(n,t)$, whose distribution tends to a semi-circular one, it is natural to expect a generalization of this interpretation of the moments in terms of paths for the moments of the $M_r(t)$'s. 
We show here that the mixed moments and free cumulants of the family $(M_r)_{r \geq 1}$ count lattice paths with general jump size, as follows. Consider an integer vector $r= (r_1, \ldots, r_p)$. For $k \geq 1$, define $\Delta(k) = \{k, k-2, k-4, \ldots, -k+2, -k\} = \{t-s; s,t \in \N, s+t = k\} \subset \Z$. We define a {\it Dyck $r$-path}  to be a function $\gamma:\{0, 1, \ldots, p\} \to \Z$ such that $\gamma(0) = 0$, $\gamma(p) = 0$, $\gamma(i) + \gamma(i-1) \geq r_i$ and $\gamma(i) - \gamma(i-1) \in \Delta(r_i)$ for all $i \in \{1, \ldots, p\}$ ($\Delta(k)$ is somehow the set of  admissible jumps for these paths). We denote by $\Gamma(r)$ the set of Dyck $r$-paths and we also consider its subset $\Gamma^*(r)$ of \emph{irreducible} Dyck $r$-paths: a Dyck $r$-path $\gamma$ is said to be irreducible if it has the property that it does not contain strictly smaller Dyck $s$-paths, in the following sense: there is no pair $(x,y) \neq (0,p)$ such that the path $\bar \gamma:\{0, \ldots y-x\} \to \Z$ defined by $\bar \gamma(i) = \gamma(x+i) - \gamma(x)$ is a Dyck $(r_{x+1}, r_{x+2}, \ldots, r_y)$-path.

It can be easily seen that Dyck $r$-paths are always positive ($\gamma(i) \geq 0$, for all $i \in \{0, \ldots, p\}$) and that the first and the last jumps are the largest, respectively smallest, possible: $\gamma(1) = r_1$ and $\gamma(p-1) = r_p$. By the following proposition, Dyck $r$-paths (resp. irreducible ones) are counted by the moments (resp. free cumulants) of the family $(M_r)_r:=(M_r(1))_r$:
\begin{pr}\label{pr:paths}
The sets $NC_2(r)$ and $\Gamma(r)$ are in bijection. The same holds true for $NC^*_2(r)$ and $\Gamma^*(r)$. In particular, we have
\[\phi(M_{r_1}M_{r_2}\cdots M_{r_p}) = \card \Gamma(r)\]
and
\[\k_p(M_{r_1},M_{r_2},\ldots, M_{r_p}) = \card \Gamma^*(r).\] 
\end{pr}
\begin{proof}
Consider a non-crossing pairing $\pi \in NC_2(r)$. We begin by constructing the path of $\Gamma(r)$ associated to $\pi$. An element $k$ of $\{1, \ldots, |r|\}$ is said to be an \emph{opener} (for $\pi$) if it appears first in its block (pair) of $\pi$. Otherwise, it is called a \emph{closer}. For $1 \leq i \leq p$, let $B_i = f_r^{-1}(i)$. As $\pi$ is non-crossing and it does not contain pairs with both ends in $B_i$, the closers appear before the openers in each $B_i$. Let $s_i$ be the number of closers of  $B_i$ and $t_i$ be the number of openers of $B_i$. We have $s_i + t_i = r_i$. Define $\gamma:\{0, 1, \ldots, p\} \to \Z$ by $\gamma(0) = \gamma(p) = 0$ and $\gamma(i) - \gamma(i-1) = t_i - s_i$, for all $1 \leq i \leq p$; we have thus $\gamma(i) - \gamma(i-1) \in \Delta(r_i)$. The value of $\gamma(i)$ is the number of \emph{open} pairs after the first $i$ groups of $\pi$. Hence, for all $i \geq 1$, $\gamma(i-1) - s_i \geq 0$. This implies $\gamma(i) + \gamma(i-1) \geq r_i$, and thus $\gamma$ is an $r$-path. In order to prove the other direction, note that a pairing $\pi \in NC_2(r)$ can be reconstructed by knowing only the number of openers/closers in each block $B_i$. This information can easily be deduced from an $r$-path $\gamma$.

The proof that the construction above is a bijection between the set of irreducible $r$-paths $\Gamma^*(r)$ and $NC_2^*(r)$ is cumbersome; we shall just give the main idea. Again, let $\pi$ be a pairing of $NC_2(r)$. The condition $\pi \vee \hat 1_r = \hat 1_{|r|}$ amounts to the fact that the standard graphical representation of  $\pi$ and $\hat 1_r$ on the same figure ($\hat 1_r$ can drawn by connecting the points of each of its groups by horizontal lines) is a connected graph. If it is not the case, then the sub-graph of a connected component corresponds to a strictly smaller $r$-path in the path $\gamma$  previously associated to $\pi$. 
\end{proof}

\begin{rk}
Note that for $r=\langle 1 \rangle_p$, $\Delta(1)=\{\pm 1\}$, and we recover the usual Dyck paths. For $r=\langle 2 \rangle_p$, and $p \geq 2$, it can easily be seen that $\Gamma^*(\langle 2 \rangle_p) = \{(0, 2, 2, \ldots, 2, 0)\}$, and we obtain the free cumulants of the centered Marchenko-Pastur (or free Poisson) distribution. 
\end{rk}

\subsection{A Toeplitz algebra model for $(M_r(1))_{r \geq 1}$}

In this section we provide a concrete realization of the family $(M_r(1))_{r \geq 1}$ 
by Toeplitz operators. Consider the Toeplitz algebra $\mathcal T$ of bounded linear operators on $\ell^2(\N)$ with its vacuum state $\omega(T) = \scalar{e_0}{T e_0}$. The shift operators are denoted by $S$ and $S^*$. Let $T_0 = 1$ and define, for all $r \geq 1$ the operators
\[T_r = \sum_{k=0}^{r}\underbrace{SS\cdots S}_{r-k \text{ times}}\underbrace{S^*S^*\cdots S^*}_{k \text{ times}} = S^r + S^{r-1}S^* + \cdots + S^{*r}.\]

It can be easily checked that the operators $T_r$ verify the recurrence relation of the (second kind) Chebyshev polynomials $T_1T_r = T_{r-1} + T_{r+1}$. It is well known that, under the vacuum state, the operator $T_1 = S + S^*$ has a semicircular distribution, and thus it has the same law as $M_1(1)$. We conclude that

\begin{pr}
The families $(T_r)_r \in (\mathcal T, \omega)$ and $(M_r(1))_r \in (\mcA, \phi)$ have the same distribution.
\end{pr}

\begin{rk} Note that we can also  realize 
 the whole family $(M_r(t))_{r \geq 1, t\in [0,+\infty)}$ on the full Fock space of the Hilbert space $L^2([0,+\infty),dx)$ with the operators (here, $\ell$ denotes the creation operator)
\[T_r(t) = \sum_{k=0}^{r}\underbrace{\ell(\indic_{[0,t)})\cdots \ell(\indic_{[0,t)})}_{r-k \text{ times}}\underbrace{\ell^*(\indic_{[0,t)})\cdots \ell^*(\indic_{[0,t)})}_{k \text{ times}} \in \mathcal B (\mathcal F(L^2([0,+\infty),dx))).\]
\end{rk}

It can be insightful to look at the matrix representations of the operators $T_r$. It can be easily verified that the $(i,j)$ coefficient of $T_r$, $T_r(i,j) = \scalar{ e_i}{T_r e_j}$ is null, unless
\begin{itemize}
\item $j-i \in \Delta(r) = \{r, r-2, \ldots, -r\}$ and
\item $j+i \geq r$,
\end{itemize}
in which case it equals $1$.

This matrix point of view introduces the connection with the set $\Gamma(r)$:
\[\phi(M_{r_1}M_{r_2}\cdots M_{r_p}) = \omega(T_{r_1}T_{r_2}\cdots T_{r_p}) = [T_{r_1}T_{r_2}\cdots T_{r_p}](0,0) = \]
\[ = \sum_{ i_0=0, i_1, \ldots, i_p = 0} T_{r_1}(i_0, i_1)T_{r_2}(i_1, i_2)\cdots T_{r_p}(i_{p-1}, i_p).\]
In order for the general term of the above sum to be non-zero, it has to be that each factor is 1, and that amounts to the fact that $\gamma = (i_0, i_1, \ldots, i_p) \in \Gamma(r)$,
providing a connection with the generalized Dyck paths discussed earlier.

\subsection{Non-commutative invariants and semi-standard Young tableaux}

In this section we show that the combinatorics of the family $(M_r)_r$ is related to semi-standard Young tableaux, which have been shown to count the number of non-commutative classical invariants of binary forms \cite{teranishi}. Here, we prove only a combinatorial result; whether there is a more profound reason for this, we ignore at this moment and connections with the representation theory of $SL_2(\C)$, $GL(n)$ or $S_n$ are to be explored.

Start by fixing a vector $r=(r_1, \ldots, r_p)$ such that $|r|$ is even and consider the Young diagram with 2 rows and $|r|/2$ columns associated to the partition $\lambda = (|r|/2, |r|/2)$ of $|r|$. A semi-standard Young tableau of shape $\lambda$ and weight $r$ is a numbering of the Young diagram of shape $\lambda$ with $r_1$ 1's, $r_2$ 2's, $\ldots$, $r_p$ $p$'s such that the rows are not decreasing and the columns are increasing. Let $c(r)$ be the number of such semi-standard Young tableaux. 

\begin{pr}
$c(r) = \card NC_2(r)$.
\end{pr}
\begin{proof}
We shall construct a bijection between the set of non-crossing pairings of $NC_2(r)$ and the set of semi-standard Young tableaux of weight $r$. Start with a pairing $\pi \in NC_2(r)$. We shall add numbers in the empty Young diagram group by group. When we arrive at the $i$-th group of $\pi$, start by appending $s_i$ $i$'s to the second row, corresponding to the $s_i$ closing pairs of the $i$-th group. Then add the remaining $t_i$ $i$'s to the top row - these are the $t_i$ opening pairs. In this way we are sure to get a row non-decreasing numbering. The fact that the columns are increasing follows from the fact that at each moment, the number of opened pairs of $\pi$ is larger or equal than the number of closed pairs. Thus the top row is always more occupied then the bottom row. 
\end{proof}
\begin{rk}
As we did for the paths, we can prove a bijection between $NC_2^*(r)$ and a strict subset of semi-standard Young tableaux. However, this is stricter than the notion of ``indecomposable'' Young tableaux, defined in \cite{teranishi}.
\end{rk}

\end{document}